\documentclass[12pt]{amsart}

\theoremstyle{plain}
\newtheorem{thm}{Theorem}[section]

\newtheorem{lemma}[thm]{Lemma}

\newtheorem{proposition}[thm]{Proposition}
\theoremstyle{definition}
\newtheorem{remark}[thm]{Remark}

\numberwithin{equation}{section}
\newcommand{\sA}{{\mathcal A}}

\newcommand{\sI}{{\mathcal I}}

\newcommand{\sL}{{\mathcal L}}

\newcommand{\sO}{{\mathcal O}}

\title[Projective normality of abelian varieties]{Buser-Sarnak invariant and projective normality of abelian varieties }
\author[Jun-Muk Hwang and Wing-Keung To]{Jun-Muk Hwang${}^1$ and Wing-Keung To${}^2$}
\address{Jun-Muk Hwang, Korea Institute for Advanced Study, Hoegiro 87, Seoul, 130-722, Korea} \email{jmhwang@kias.re.kr}
\address{Wing-Keung To,
Department of Mathematics, National University of Singapore, 2
Science Drive 2, Singapore 117543} \email{mattowk@nus.edu.sg}
\thanks{${}^1$ supported by Basic Science Research Program (ASARC) through NRF funded
by MEST (2009-0063180)}
\thanks{${}^2$ supported partially by the research grant R-146-000-106-112 from the National University of Singapore and the Ministry of Education}
\begin{document}

\maketitle \numberwithin{equation}{section}
\bigskip
\noindent {\bf Abstract.} {\it We show that a general $n$-dimensional polarized abelian variety $(A,L)$ of a given polarization
type and satisfying $\displaystyle h^0(A, L) \geq \dfrac{8^n}{2}
\cdot \dfrac{n^n}{ n ! }$ is projectively normal.  In the process, we also obtain a sharp lower bound for the volume of a purely one-dimensional complex analytic subvariety in a geodesic tubular neighborhood of a subtorus of a compact complex torus.}



\bigskip
\section{Introduction and Statement of Results}\label{Section 1}

\medskip
Let $A$ be an abelian variety of dimension $n$, and let $L$ be an
ample line bundle over $A$.  Such a pair $(A,L)$ is called a {\it
polarized abelian variety}.  We are interested in studying the
projective normality of $(A,L)$, which plays an important role in
the theory of linear series associated to $(A,L)$.  For each
$r\geq 1$, we consider the multiplication map
\begin{equation}\label{1.1}\rho_r:Sym^r H^0(A,L)\to H^0(A,L^{\otimes
r})
\end{equation}
induced by $(\sigma_1,\cdots,\sigma_r)\to \sigma_1\cdots\sigma_r$ for $\sigma_1,\cdots,\sigma_r\in H^0(A,L)$.
Here $Sym^rH^0(A,L)$ denotes the $r$-fold symmetric tensor power of
$H^0(A,L)$.  Recall that $(A,L)$ (or simply $L$) is said to be {\it
projectively normal} if $\rho_r$ is surjective for each $r\geq 1$.
The projective normality of a polarized abelian variety $(A,L)$ is
well-understood in the case when $L$ is not primitive, i.e., when
there exists a line bundle $L^\prime$ such that $L =
L^{\prime\otimes m}$ for some integer $m \geq 2$ (cf. the references
in \cite{Iy}). However, not much is known for the case when $L$ is
primitive.

\medskip
In the primitive case, the main interest is to find conditions on
the polarization type $d_1|d_2| \cdots| d_n$ of $(A,L)$ or on
$h^0(A,L):=\dim_{\mathbb C}H^0(A,L) $ (note that $h^0(A,L)=d_1\cdots
d_n$) which will guarantee the projective normality of a {\it
general} $(A,L)$ of a given polarization type.  Along this line, J.
Iyer \cite{Iy} proved the following result:

\begin{thm}\label{Theorem 1.1} ([Iy, Theorem 1.2]) Let $(A,L)$ be a
polarized simple abelian variety of dimension $n$.  If $h^0(A,L)>2^n
n!$, then $L$ is projectively normal.
\end{thm}

See also \cite{FG} for related results in the lower dimensional cases
when $n=3,4$. These works use the theory of theta functions and
theta groups.

\medskip
Our goal is to relate this problem to the Buser-Sarnak invariant
$m(A,L)$ of the polarized abelian variety (cf. [L2, p.291]).  Since $A$ is a
compact complex torus, one may write $A=\mathbb C^n/\Lambda$, where
$\Lambda$ is a lattice in $\mathbb C^n$.  It is well-known that
there exists a unique translation-invariant flat K\"ahler form
$\omega$ on $A$ such that $c_1(L)=[\omega]\in H^2(A,\mathbb Z)$. The
real part of $\omega$ gives rise to an inner product
$\langle~,~\rangle$ on $\mathbb C^n$, and we denote by $\Vert~\Vert$
the associated norm on $\mathbb C^n$.  The Buser-Sarnak invariant is
given by
\begin{equation}\label{1.2}m(A,L):=
\min_{ \lambda \in \Lambda\setminus\{0\} }\Vert\lambda\Vert^2.
\end{equation}
In other words, $m(A,L)$ is the square of the minimal length of a
non-zero lattice vector in $\Lambda$ with respect to
$\langle~,~\rangle$.  The study of this invariant was initiated by
Buser and Sarnak in \cite{BS}, where they studied it
 for principally polarized abelian varieties and Jacobians. In particular, they showed the existence of a principally polarized abelian
variety $(A,L)$ with
\begin{equation}\label{1.3}m(A, L) \geq\dfrac{1}{\pi}
\root n \of {2 L^n}.
\end{equation}
In [Ba], Bauer
generalized this to abelian varieties of arbitrary polarization type
(cf. [L2, p. 292-293]).

\medskip
The relevance of  the invariant $m(A,L)$ in the study of
algebro-geometric questions was first observed by Lazarsfeld
\cite{L1}, where he obtained a lower bound for the Seshadri number
of $(A,L)$ in terms of $m(A,L)$ (cf. [L2, p. 293]).  In
particular, $m(A,L)$ gives information on generation of jets by
$H^0(A,L)$. Furthermore, Bauer used the existence of $(A,L)$
satisfying (\ref{1.3}) together with Lazarsfeld's above result to
obtain the following result:

\begin{thm}\label{Theorem 1.2} ([Ba, Corollary
2]) Let $(A,L)$ be a general $n$-dimensional polarized abelian
variety of a given polarization type.  If $\displaystyle h^0(A,L)
\geq \frac{8^n}{2} \cdot \frac{n^n}{n!}$, then $L$ is very ample.
\end{thm}

\medskip
Now we state our main result in this
paper as follows:

\begin{thm}\label{Theorem 1.3}
A general $n$-dimensional polarized abelian variety $(A,L)$ of a given polarization
type and satisfying $\displaystyle h^0(A, L) \geq \dfrac{8^n}{2}
\cdot \dfrac{n^n}{ n ! }$ is projectively normal.
\end{thm}

\medskip
Using Stirling's formula ($n!\sim \sqrt{2\pi}\,n^{n+\frac
12}e^{-n}$), one easily sees that our bound in Theorem
\ref{Theorem 1.3} improves Iyer's bound in Theorem \ref{Theorem
1.1} substantially for large $n$. Note that our bound in Theorem
\ref{Theorem 1.3} for projective normality is the same as Bauer's
bound in Theorem \ref{Theorem 1.2} for very ampleness. To our
knowledge, this is just a coincidence. Although the proofs of both
theorems use Bauer's generalization of (\ref{1.3}), Theorem
\ref{Theorem 1.2} itself is not used in the proof of Theorem
\ref{Theorem 1.3}.
 Finally it is worth
comparing Theorem \ref{Theorem 1.3} with the result in \cite{FG}
and \cite{Ru} that there is a polarization type $d_1|d_2| \cdots|
d_n$ with $\displaystyle d_1\cdots d_n=h^0(A,L) = \dfrac{4^n}{2}$
such that no abelian varieties of this polarization type is
projectively normal.

\medskip
We describe briefly our approach as follows.  First we obtain an
auxiliary result, which is a sharp lower bound for the volume of a
purely one-dimensional complex analytic subvariety in a geodesic tubular
neighborhood of a subtorus of a compact complex torus (see Proposition \ref{Proposition 2.3} for the
precise statement).  As a consequence, we obtain a lower bound of
the Seshadri number of the line bundle $p_1^*L \otimes p_2^*L$ along
the diagonal of $A\times A$ in terms of $m(A,L)$ (see Proposition \ref{3.2}).
Here $p_i:A\times A\to A$ denotes the projection onto the $i$-th
factor, $i=1,2$.   We believe that these two auxiliary results are
of independent interest beside their application to the projective
normality problem. Finally the proof of Theorem \ref{Theorem 1.3}
involves the use of the second auxiliary result and applying Bauer's
result mentioned above in (\ref{1.3}).

\bigskip
\section{Volume of subvarieties near a complex subtorus}\label{Section 2}

\medskip
In this section, we are going to obtain a sharp lower bound for
the volume of a purely $1$-dimensional complex analytic subvariety
in a tubular open neighborhood of a subtorus of a compact complex
torus (see Proposition \ref{Proposition 2.3}). This inequality is
inspired by an analogous inequality in the hyperbolic setting
proved in \cite{HT}. The proof of the current case is much simpler
than the one in \cite{HT}, using a simple projection argument and
Federer's volume inequality for analytic subvarieties in a
Euclidean ball in $\mathbb C^n$ (cf. e.g. [St] or [L2, p. 300]).

\medskip
Let $T=\mathbb C^n/\Lambda$ be an $n$-dimensional compact complex torus associated to a lattice $\Lambda\subset \mathbb C^n$
and endowed with a flat translation-invariant K\"ahler form
$\omega$. For simplicity, we call $(T,\,\omega)$ a {\it polarized
compact complex torus}.  Let $\langle~,~\rangle$ and $\Vert~\Vert$
be the inner product and norm on $\mathbb C^n$ associated to $\omega$ as in
Section \ref{Section 1}.  Next we let $S$ be a $k$-dimensional compact complex
subtorus of $T$, where $0\leq k<n$. It is well-known that $S$ is the
quotient of a $k$-dimensional linear subspace $F\cong \mathbb C^k$
of $\mathbb C^n$ by a sublattice $\Lambda_S\subset\Lambda$ of rank
$2k$ and such that $\Lambda_S=\Lambda\cap F$.  Let $F^{\perp}$ be
the orthogonal complement of $F$ in $\mathbb C^n$ with respect to
$\langle~,~\rangle$, and let $q_F:\mathbb C^n\to F$ and
$q_{F^\perp}:\mathbb C^n\to F^{\perp}$ denote the associated unitary
projection maps. Similar to (\ref{1.2}), we define the {\it relative
Buser-Sarnak invariant} $m(T,\,S,\,\omega)$ given by
\begin{equation}\label{2.1}m(T,\,S,\,\omega):=\min_{\lambda\in\Lambda\setminus\Lambda_S}\Vert
q_{F^{\perp}}(\lambda)\Vert^2.
\end{equation}
In other words, $m(T,\,S,\,\omega)$ is the
square of the minimal distance of a vector in
$\Lambda\setminus\Lambda_S$ from the linear subspace $F$.

\medskip\noindent
\begin{remark}\label{Remark 2.1}
(i) The invariant $m(A,L)$ in (\ref{1.2}) corresponds to the special case when $S=\{0\}$ and $[\omega]=c_1(L)$, i.e., $m(A,L)=m(A,\{0\},\omega)$.
\par\noindent
(ii)  From the discreteness of $\Lambda$, the equality $\Lambda_S=\Lambda\cap F$ and the compactness of $S=F/\Lambda_S$, one easy checks that $m(T,\,S,\,\omega)>0$ and its value is attained by some $\lambda\in\Lambda\setminus\Lambda_S$.
\end{remark}

\medskip
With
regard to the Riemannian geometry associated to $\omega$, one also
easily sees that the geodesic distance function $d_T:T\times T\to
\mathbb R$ of $T$ with respect
to $\omega$ can be expressed in terms of $\Vert~\Vert$ given
by
\begin{equation}\label{2.2}
d_T(x,y)=\inf \{ \Vert z-w\Vert\,\big|\, p(z)=x,\,p(w)=y\},
\end{equation}
where $p:\mathbb C^n\to T$ denotes the covering projection map.
For any given $r>0$, we consider
the open subset of $T$ given by
\begin{equation}\label{2.3}W_r:=\{x\in
T,\big|\, d_T(x,S)<r\}\supset S,
\end{equation}
where as usual,
\begin{equation}\label{2.4}
d_T(x,S):=\inf_{y\in S}d_T(x,y)=
\min\{\Vert q_{F^{\perp}}(z)\Vert\,\big|\, p(z)=x\}
\end{equation}
(note that the second equality in (\ref{2.4}) follows from standard facts on inner product spaces, and as in Remark \ref{Remark 2.1}, the minimum value in the last expression in (\ref{2.4}) is attained by some $z$).  We simply call $W_r$ the {\it geodesic tubular
neighborhood of $S$ in $T$} of radius $r$. Next we consider the
biholomorphism $\widetilde{\phi}:F\times F^{\perp}\to \mathbb C^n$ given by
\begin{equation}\label{2.5}\widetilde{\phi}(z_1,z_2)=z_1+z_2\quad\textrm{for }(z_1,z_2)\in F\times F^{\perp},\end{equation}
It is easy to see that the covering projection map $p\circ\widetilde{\phi}:F\times F^{\perp}\to T$
is equivariant under the action of $\Lambda_S$ on $F\times F^{\perp}$ given by $(z_1,z_2)\to (z_1+\lambda,z_2)$ for
$(z_1,z_2)\in F\times F^{\perp}$ and $\lambda \in \Lambda_S$.  It follows readily that $p\circ\widetilde{\phi}$ descends to a well-defined covering projection map denoted by $\phi:S\times
F^{\perp}\to T$ (in particular, $\phi$ is a local biholomorphism).
Consider the flat translation-invariant K\"ahler form on $\mathbb
C^n$ given by
\begin{equation}\label{2.6}
\omega_{\mathbb
C^n}:=\frac{\sqrt{-1}}{2}\partial\overline{\partial}\Vert
z\Vert^2,\quad z\in\mathbb C^n,
\end{equation}
which is easily seen to descend to the K\"ahler form $\omega$ on
$T$. Consider also the flat K\"ahler form on $F^{\perp}$ given by
$\omega_{F^{\perp}}:=\omega_{\mathbb C^n}\big|_{F^{\perp}}$, and for
any $r>0$, let $B_{F^{\perp}}(r):=\{z\in F^{\perp}\,\big|\, \Vert
z\Vert<r\}$ denote the associated open ball of radius $r$. Let
$\omega_S:=\omega\big|_S$.  Note that $\phi\big|_{S\times\{0\}}$ is
given by the identity map on $S$.  It admits
biholomorphic extensions as follows:

\medskip\noindent
\begin{lemma}\label{Lemma 2.2}
For any real number $r$ satisfying
$0<r\leq\frac{\sqrt{m(T,\,S,\,\omega)}}{2}$, one has a biholomorphic isometry
\begin{equation}\label{2.7}
\phi_r:(S,\,\omega_S)\times
(B_{F^{\perp}}(r),\,\omega_{F^{\perp}}\big|_{B_{F^{\perp}}(r)})
\to(W_r,\, \omega\big|_{W_r})
\end{equation}
given by $\phi_r:=\phi\big|_{S\times B_{F^{\perp}}(r)}$.
\end{lemma}
\medskip\noindent
{\it Proof.}
First we fix a real number $r$ satisfying
$0<r\leq\frac{\sqrt{m(T,\,S,\,\omega)}}{2}$.  From (\ref{2.3}), (\ref{2.4}) and the obvious identity $q_{F^{\perp}}(\widetilde{\phi}(z_1,z_2))=z_2$ for $(z_1,z_2)\in F\times F^{\perp}$, one easily sees that $\phi(S\times B_{F^{\perp}}(r))\subset W_r$, and thus the map $\phi_r$ in (\ref{2.7}) is well-defined.  For each $x\in W_r$, it follows from the second equality in (\ref{2.4}) that there exists $z\in \mathbb C^n$ such that $p(z)=x$ and $\Vert q_{F^{\perp}}(z)\Vert=d_T(x, S)<r$.  Now, $q_F(z)$ descends to a point $x_S$ in $S$, and one easily sees that $\phi_r(x_S, q_{F^{\perp}}(z))=x$ with $(x_S, q_{F^{\perp}}(z))\in S\times B_{F^{\perp}}(r)$.  Thus $\phi_r$ is surjective.  Next we are going to prove by contradiction that $\phi_r$ is injective.  Suppose $\phi_r$ is not injective.  Then it implies readily that there exist
two points $(z_1,z_2)$, $(z_1^\prime,z_2^\prime)\in  F\times B_{F^{\perp}}(r)$ such that
\par (i) either $z_1-z_1^\prime\notin\Lambda_S$ or $z_2\neq z_2^\prime$; and
\par (ii) $z_1+z_2-(z_1^\prime+z_2^\prime)=\lambda$ for some $\lambda\in \Lambda$
\par\noindent
(here (i) means that $(z_1,z_2)$, $(z_1^\prime,z_2^\prime)$ descend to two different points in $S\times B_{F^{\perp}}(r))$.  In both cases in (i), one easily checks that $\lambda\in\Lambda\setminus\Lambda_S$.  On the other hand, one also sees from (ii) that
$q_{F^{\perp}}(\lambda)=z_2-z_2^\prime$ and thus
\begin{equation}\label{2.8}
\Vert
q_{F^{\perp}}(\lambda)\Vert\leq \Vert z_2\Vert+\Vert z_2^\prime\Vert<r+r=2r\leq\sqrt{m(T,\,S,\,\omega)},
\end{equation}
which contradicts the definition of $m(T,\,S,\,\omega)$ in (\ref{2.1}).  Thus, $\phi_r$ is injective, and we have proved that $\phi_r$ is a bihomorphism.
Finally from the obvious identity $\Vert z_1\Vert^2+\Vert z_2\Vert^2=\Vert z_1+z_2\Vert^2$ for $(z_1,z_2)\in F\times F^{\perp}$, and upon taking $\frac{\sqrt{-1}}{2}\partial\overline\partial$, one easily sees that $\widetilde{\phi}:(F,\,\omega_F)\times (F^{\perp},\,\omega_{F^{\perp}})\to(\mathbb
C^n,\,\omega_{\mathbb
C^n})$ is a biholomorphic isometry (cf. (\ref{2.6})).  It follows readily that the induced covering projection map $\phi:(S,\,\omega_S)\times (F^{\perp},\,\omega_{F^{\perp}})\to (T,\,\omega)$ is a local isometry.  Upon restricting $\phi$ to $S\times B_{F^{\perp}}(r)$, one sees that the biholomorphism $\phi_r$ is an isometry.
\qed

\medskip
For each $x\in S$ and each non-zero holomorphic tangent vector
$v\in {\rm T}_{x,T}$ orthogonal to ${\rm T}_{x,S}$, it is easy to
see that there exists a unique 1-dimensional totally geodesic
(flat) complex submanifold $\mathcal \ell$ of
$W_{\frac{\sqrt{m(T,\,S,\,\omega)}}{2}}$ passing through $x$ and
such that ${\rm T}_{x,\mathcal \ell}=\mathbb Cv$. We simply call
such $\mathcal \ell$ an {\it $S$-orthogonal line} of
$W_{\frac{\sqrt{m(T,\,S,\,\omega)}}{2}}$.  For a complex analytic
subvariety $V$ in an open subset of $T$, we simply denote by
$\textrm{Vol}\,(V)$ its volume with respect to the K\"ahler form
$\omega$, unless otherwise stated.  It is easy to see that for
each $0<r\leq\frac{\sqrt{m(T,\,S,\,\omega)}}{2}$, the values of
$\textrm{Vol}\,(\mathcal \ell\cap W_r)$ are the same for all the
$S$-orthogonal lines $\mathcal \ell$ in
$W_{\frac{\sqrt{m(T,\,S,\,\omega)}}{2}}$.  As such, $
\textrm{Vol}\,(\mathcal \ell\cap {W_r})$ is an unambiguously
defined number depending on $r$ only (cf. (\ref{2.9}) below). Next
we consider the blow-up $\pi:\widetilde{T}\to T $ of $T$ along
$S$, and denote the associated exceptional divisor by
$E:=\pi^{-1}(S)$.  For a complex analytic subvariety $V$ in an
open subset of $T$ such that $V$ has no component lying in $S$, we
denote its strict transform with respect to $\pi$ by $\widetilde{
V}:=\overline{\pi^{-1}(V\setminus S})$. As usual, for an $\mathbb
R$-divisor $\Gamma$ and a complex curve $C$ in a complex manifold,
we denote by $\Gamma \cdot C$ the intersection number of $\Gamma$
with $C$.  Our main result in this section is the following

\begin{proposition}\label{Proposition 2.3}
Let $(T,\,\omega)$ a polarized compact complex torus of dimension
$n$, and let $S$ be a $k$-dimensional compact complex subtorus of $T$, where $0\leq k<n$.  Let
$\pi:\widetilde{T}\to T $ be the blow-up of $T$ along $S$ with the
exceptional divisor $E=\pi^{-1}(S)$ as above.  Then for any real
number $r$ satisfying $0<r\leq\frac{\sqrt{m(T,\,S,\,\omega)}}{2}$
and any purely $1$-dimensional complex analytic subvariety $V$ of
the geodesic tubular neighborhood $W_r$ of $S$ such that $V$ has no component lying in $S$, one has
\begin{eqnarray}\label{2.9}
\textrm{Vol}\,(V)&\geq& \pi r^2\cdot(\widetilde V\cdot E)\\
~&=& \textrm{Vol}\,(\mathcal \ell\cap {W_r})\cdot(\widetilde V\cdot
E).\nonumber
\end{eqnarray}
In particular, for each $0<r\leq\frac{\sqrt{m(A,\,S,\,\omega)}}{2}$
and each non-negative value $s$ of $\widetilde V\cdot E$, the lower
bound in (\ref{2.9}) is attained by the volume of some (and hence
any) $V$ consisting of the intersection of ${W_r}$ with the union of
$s$ copies of $S$-orthogonal lines counting multiplicity.
\end{proposition}

\medskip\noindent
{\it Proof.}
Let $V\subset\widetilde{W_r}$ be as above.
It is clear that Proposition \ref{Proposition 2.3} for the general case when $V$ is reducible follows from the special case when $V$ is irreducible, and that (\ref{2.9}) holds trivially for the case when $V\cap S=\emptyset$.  As such, we will assume without loss of generality that
\begin{equation}\label{2.10}
V\textrm{ is irreducible},\quad V\cap
S\neq\emptyset\quad\textrm{and} \quad V\not\subset
S.\end{equation} Then $\widetilde V\cap E$ consists of a finite
number of distinct points $y_1,\cdots,y_\kappa$ with intersection
multiplicities $m_1,\cdots,m_\kappa$ respectively, so that
\begin{equation}\label{2.11}\widetilde V\cdot E=m_1+\cdots+m_\kappa.
\end{equation}
By Lemma \ref{Lemma 2.2}, we have a biholomorphic isometry
\begin{equation}\label{2.12}
(W_r,\omega\big|_{W_r})\cong(S\times F^{\perp}(r), \,\eta_1^*\omega_S+\eta_2^*\omega_{F^{\perp}}).
\end{equation}
Here $\eta_1:S\times F^{\perp}(r)\to S$ and $\eta_2:S\times F^{\perp}(r)\to F^{\perp}(r)$ denote the projections onto the first and second factor respectively.  Next we make an identification $F^{\perp}\cong\mathbb C^{n-k}$ with Euclidean coordinates $z_1,z_2,\cdots,z_{n-k}$ associated to an orthonormal basis of $(F^{\perp}, \langle~,~\rangle\big|_{F^{\perp}})$.  Under this identification, we have
\begin{eqnarray}\label{2.13}
\qquad F^{\perp}(r)&=&\{z=(z_1,z_2,\cdots,z_{n-k})\in\mathbb C^{n-k}\, \big|\, |z|<r\},\quad\textrm{and}\\
\nonumber \omega_{F^{\perp}}&=&\frac{\sqrt{-1}}{2}\sum_{i=1}^{n-k}
dz_i\wedge d\overline{z}_i.
\end{eqnarray}
Here $|z|=\sqrt{\sum_{i=1}^{n-k}|z_i|^2}$. Note that $\eta_2$ (and
thus also $\eta_2\big|_V$) is a proper holomorphic mapping, and
thus by the proper mapping theorem, $V^\prime:=\eta_2(V)$ is a
complex analytic subvariety of $F^{\perp}(r)$.  From (\ref{2.10}),
one easily sees that $V^\prime $ is irreducible and of pure
dimension one, and $\eta_2\big|_V:V\to V^\prime$ is a
$\delta$-sheeted branched covering for some $\delta\in\mathbb N$.
Note that $0\in V^\prime$ since $V\cap S\neq\emptyset$, and we
denote by $\mu$ the multiplicity of $V^\prime$ at the origin $0\in
F^{\perp}(r)$.  Let $[V]$ (resp. $[V^\prime]$) denote the closed
positive current defined by integration over $V$ (resp.
$V^\prime$) in $W_r$ (resp. $F^{\perp}(r)$).  Then via the
identifications in (\ref{2.13}), it follows from Federer's volume
inequality for complex analytic subvarieties in a complex
Euclidean ball (see e.g. [St] or [L2, p. 300]) that one has
\begin{equation}\label{2.14}
\int_{F^{\perp}(r)}[V^\prime]\wedge \omega_{F^{\perp}}\geq\mu
\cdot\pi r^2.
\end{equation}
Next we consider a linear projection map $\psi:F^{\perp}\to\mathbb
C$ from $F^{\perp}$ onto some one-dimensional linear subspace
(which we identify with $\mathbb C$).  It follows readily from the
definition of $\mu$ that for a generic $\psi$,
$\psi\big|_{V^\prime}:V^\prime\to \psi(V^\prime)$ is an
$\mu$-sheeted branched covering.  Furthermore, by considering the
local description of the blow-up map $\pi$ (cf. e.g. [GH, p.
603]), one easily sees that for each $y_j\in\widetilde V\cap E$,
$1\leq j\leq\kappa$, there exists an open neighborhood $U_j$ of
$y_j$ in $\widetilde V$ such that for a generic $\psi$, the
function $\psi\circ \eta_2\circ\pi\big|_{U_j}:U_j\to\mathbb C$ is
a defining function for $E\cap U_j$ in $U_j$, so that $\psi\circ
\eta_2\circ\pi\big|_{\widetilde V\cap U_j}$ is an $m_j$-sheeted
branched covering onto its image, shrinking $U_j$ if necessary.
Thus by considering the degree of the map $\psi\circ
\eta_2\circ\pi\big|_{\widetilde V}$ for a generic $\psi$, one gets
\begin{equation}\label{2.15}
\delta\cdot\mu=m_1+\cdots +m_\kappa.
\end{equation}
Under the identification in (\ref{2.12}), we have
\begin{eqnarray}\label{2.16}
\textrm{Vol}(V)&=&\int_{W_r}[V]\wedge\omega\\
&=&\int_{S\times F^{\perp}(r)} [V]\wedge (\eta_1^*\omega_S+\eta_2^*\omega_{F^{\perp}})\nonumber \\
&\geq & \int_{S\times F^{\perp}(r)} [V]\wedge \eta_2^*\omega_{F^{\perp}}\quad(\textrm{since }\eta_1^*\omega_S\geq 0)\nonumber \\
&=& \delta \int_{ F^{\perp}(r)}  [V^\prime]\wedge \omega_{F^{\perp}}\nonumber \\
&~&\quad(\textrm{upon taking the direct image }\eta_2^*)\nonumber \\
&\geq& \delta\cdot\mu\cdot\pi r^2\quad(\textrm{by (\ref{2.14})})\nonumber \\
&=&\pi r^2\cdot (\widetilde V\cdot
E)\quad(\textrm{by (\ref{2.11}) and (\ref{2.15})}),\nonumber
\end{eqnarray}
which gives the first line of (\ref{2.9}).  Next we take an $S$-orthogonal line
$\ell$ of $W_{\frac{\sqrt{m(T,\,S,\,\omega)}}{2}}$.  Then under the identifications in
(\ref{2.12}), (\ref{2.13}) and upon making a unitary change of $F^{\perp}$ if necessary,
one easily sees that $\ell\cap W_r$ can be given by $\{x\}\times \{(z_1,0,\cdots,0)\in\mathbb C^{n-k}
\big|\,|z_1|<r\}$ for some fixed point $x\in S$, and it follows readily that
\begin{equation}\label{2.17}
\textrm{Vol}(\ell\cap W_r)= \int_{|z_1|<r}\dfrac{\sqrt{-1}}{2}dz_1\wedge d\overline{z}_1=\pi r^2,
\end{equation}
which gives the second line of (\ref{2.9}).  Finally we remark that the last statement of
Proposition \ref{Proposition 2.3} is a direct consequence of (\ref{2.9}), and thus we have finished the proof of Proposition \ref{Proposition 2.3}.
\qed

\bigskip
\section{Seshadri number along the diagonal of $A\times A$}\label{Section 3}

\medskip
In this section, we let $(A=\mathbb C^n/\Lambda,\,L)$ be a polarized abelian variety of dimension $n$, and let the associated objects $\omega$,
$\langle~,~\rangle$, $\Vert~\Vert$ and $m(A,L)$ be as defined in Section \ref{Section 1}.  Next we consider the Cartesian product $A\times A$, and we denote by $p_i:A\times A\to A$ the projection map onto the $i$-th factor.  It is easy to see that $p_1^*L\otimes p_2^*L$ is an ample line bundle over the $2n$-dimensional (product) abelian variety $A\times A$, and the associated translation-invariant
flat K\"ahler form on $A\times A$ is given by $\omega_{A\times A}:=p_1^*\omega +  p_2^*\omega$.  In particular, one has
\begin{equation}\label{3.1}
[\omega_{A\times A}]=c_1(p_1^*L\otimes p_2^*L)\in H^2(A\times A,\,\mathbb Z).
\end{equation}
Furthermore, it is easy to see that the diagonal of $A\times A$ given by
\begin{equation}\label{3.2}
D:=\{(x,y)\in A\times A\,\big|\, x=y\}
\end{equation}
is an $n$-dimensional abelian subvariety of $A\times A$.  Let $m(A\times A, D,\omega_{A\times A})$ be the relative Buser-Sarnak invariant as given in (\ref{2.1}).

\medskip\noindent
\begin{lemma}\label{Lemma 3.1}
We have
\begin{equation}\label{3.3}
m(A\times A, D,\omega_{A\times A})=\dfrac{m(A,L)}{2}.
\end{equation}
\end{lemma}
\medskip\noindent
{\it Proof.}
First we write $A\times A=(\mathbb C^n\times\mathbb C^n)/(\Lambda\times\Lambda)$, and we denote by
$\langle~,~\rangle_{\mathbb C^n\times\mathbb C^n}$ and $\Vert~\Vert_{\mathbb C^n\times\mathbb C^n}$ the inner product and norm on $\mathbb C^n\times\mathbb C^n$ associated to $\omega_{A\times A}$.  It is easy to see that as a compact complex subtorus of $A\times A$, $D$ is isomorphic to the quotient $F/\Lambda_D$, where $F:=\{(z,z)\,\big|\, z\in\mathbb C^n\}\subset\mathbb C^n\times\mathbb C^n$ and $\Lambda_D:=\{(\lambda,\lambda)\,\big|\,\lambda\in\Lambda\}\subset \Lambda \times \Lambda$.
Denote by $F^{\perp}$ the orthogonal complement of $F$ in $\mathbb C^n\times\mathbb C^n$ with respect to $\langle~,~\rangle_{\mathbb C^n\times\mathbb C^n}$, and let $q_{F^{\perp}}:\mathbb C^n\times\mathbb C^n\to F^{\perp}$ be the corresponding unitary projection map.  Then for any $(\lambda_1,\lambda_2)\in
\Lambda \times \Lambda$, one easily checks that
$q_{F^{\perp}}(\lambda_1,\lambda_2)=(\frac{\lambda_1-\lambda_2}{2} ,\frac{\lambda_2-\lambda_1}{2} )$, and thus
\begin{equation}\label{3.4}
\Vert q_{F^{\perp}}(\lambda_1,\lambda_2)\Vert^2_{\mathbb C^n\times\mathbb C^n}=\Vert\frac{\lambda_1-\lambda_2}{2}\Vert^2+\Vert\frac{\lambda_2-\lambda_1}{2}\Vert^2=
\dfrac{\Vert\lambda_1-\lambda_2\Vert^2}{2}.
\end{equation}
Together with the obvious equality $\{\lambda_1-\lambda_2\big|\,(\lambda_1,\lambda_2)\in(\Lambda \times \Lambda)\setminus \Lambda_D\}=\Lambda\setminus\{0\}$ (and upon writing $\lambda=\lambda_1-\lambda_2$), one gets
\begin{equation}\label{3.5}
\inf_{ (\lambda_1,\lambda_2)\in (\Lambda \times \Lambda)\setminus \Lambda_D}   \Vert q_{F^{\perp}}(\lambda_1,\lambda_2)\Vert^2_{\mathbb C^n\times\mathbb C^n}=\dfrac{1}{2} \inf_{\lambda \in  \Lambda\setminus\{0\}}\Vert\lambda\Vert^2,
\end{equation}
which, upon recalling (\ref{1.2}) and (\ref{2.1}), gives (\ref{3.3}) immediately.
\qed

\medskip
Next we let $\pi:\widetilde{A\times A}\to A\times A$ be the blow-up of $A\times A$ along $D$ with the associated exceptional divisor given by $E:=\pi^{-1}(D)$.  We consider the line bundle $p_1^*L\otimes p_2^*L$ over $A\times A$, and denote its pull-back to $\widetilde{A\times A}$ by
\begin{equation}\label{3.6}
\sL:=\pi^*(p_1^*L\otimes p_2^*L).
\end{equation}
Then the
Seshadri number $\epsilon(p_1^*L\otimes p_2^*L, D)$ of $p_1^*L\otimes p_2^*L$ along $D$ is defined by
\begin{equation}\label{3.7}
\epsilon(p_1^*L\otimes p_2^*L, D):=\sup\{\,\epsilon\in\mathbb
R\,\big|\,\sL-\epsilon E \textrm{ is nef on }\widetilde{A\times
A}\} \end{equation}
(see e.g. [L2, Remark 5.4.3] for the
general definition and \cite{D} for its origin).  Here as usual, an $\mathbb R$-divisor $\Gamma$
on an algebraic manifold $M$ is said to be nef if
$\Gamma\cdot C\geq 0$ for any algebraic curve $C\subset M$.  Our main result in this section is the following

\begin{proposition}\label{Proposition 3.2}
Let $(A,L)$ be a polarized abelian variety of dimension $n$, and let
$\sL$ be as in (\ref{3.6}).  Then $\sL-\alpha E$ is nef on $\widetilde{A\times A}$ for all $0\leq\alpha\leq \frac{\pi}{8}\cdot m(A,L)$.  In particular, we have
\begin{equation}
\label{3.8} \epsilon(p_1^*L \otimes p_2^*L , D)\geq
\dfrac{\pi}{8}\cdot m(A,L).
\end{equation}
\end{proposition}

\medskip\noindent
{\it Proof.}  First it is easy to see from (\ref{3.6}) that $\sL$ is nef, and thus the proposition holds for the case when $\alpha=0$.  Now we fix a number $\alpha$ satisfying $0<\alpha\leq \frac{\pi}{8}\cdot m(A,L)$.  Then it is easy to see from Lemma \ref{Lemma 3.1} that $\alpha=\pi r^2$ for some $r$ satisfying $0<r\leq\frac{\sqrt{m(A\times A, D,\omega_{A\times A})}}{2}$.  For each such $r$, we let $W_r$ be the geodesic tubular neighborhood of $D$ in $A\times A$ of radius $r$ as defined in (\ref{2.3})
(with $T$ and $S$ there given by $A\times A$ and $D$ respectively).  Let $C$ be an algebraic curve in
$\widetilde{A\times A}$.  First we consider the case when $C$ is irreducible and $C\not\subset E$, so that
$\pi(C)\not\subset D$ and $C$ coincides with the strict transform of $\pi(C)$ with respect to the blow-up map $\pi$ (i.e., $C=\widetilde{\pi(C)}$ in terms of the notations in Section \ref{Section 2}).  Then by (\ref{3.1}), (\ref{3.6}) and upon taking the direct image $\pi_*$, we get
\begin{eqnarray}\label{3.9}
\sL\cdot C &=&\int_{\widetilde{A\times A}}[C]\wedge\pi^*\omega_{A\times A}\\
&=&\int_{A\times A}[\pi(C)]\wedge\omega_{A\times A}\nonumber\\
&\geq&\int_{W_r}[\pi(C)]\wedge\omega_{A\times A}\nonumber\\
&\geq&\pi r^2\cdot (E\cdot C)\quad(\textrm{by Proposition \ref{Proposition 2.3}}),\nonumber\\
&=& \alpha\cdot (E\cdot C).\nonumber
\end{eqnarray}
In other words, we have
\begin{equation}
\label{3.10} (\sL-\alpha E)\cdot C\geq 0.
\end{equation}
Next we consider the case when $C$ is irreducible and $C\subset E$.  By considering translation-invariant vector fields on $D$ and $A\times A$, one easily sees that
the normal bundle $N_{D|(A\times A)}$ is
holomorphically trivial over $D$.  It follows readily that the
line bundle $[E]\big|_E$ is isomorphic to $\sigma^*\mathcal
O_{\mathbb P^{n-1}}(-1)$, where $\sigma:D\times \mathbb P^{n-1}\to
\mathbb P^{n-1}$ denotes the projection onto the second factor.  Hence  $E\cdot C\leq 0$ for any irreducible curve
$C\subset E$. Together with the nefness of $\sL$, it follows readily that (\ref{3.10}) also holds for the irreducible case when $C\subset E$.  Finally one easily sees that (\ref{3.10}) for the case when $C$ is reducible follows readily from the case when $C$ is irreducible.  Thus we have finished the proof of the nefness of $\sL-\alpha D$ for all $0\leq\alpha\leq \frac{\pi}{8}\cdot m(A,L)$, which also leads to
(\ref{3.8}) readily.
\qed

\bigskip
\section{Projective normality}\label{Section 4}

\medskip
In this section, we are going to give the proof of  Theorem
\ref{Theorem 1.3}, and we follow the notation in Section
\ref{Section 3}.  First we have

\begin{proposition}\label{Proposition 4.1}
Let $(A,L)$, $n$, $E$ and $\sL$ be as in Proposition \ref{Proposition 3.2}.
If $\sL
\otimes \sO(- n E)$ is nef and big, then $L$ is projectively
normal.
\end{proposition}

\begin{proof} By [Iy, Proposition 2.1], one knows that
the surjectivity of the multiplication maps $\rho_r$ in
(\ref{1.1}) for all $r\geq 1$ will follow from the surjectivity of $\rho_2$ (i.e., the case when $r=2$). Thus to prove that $L$ is projectively
normal, it suffices to show that the multiplication map
\begin{equation}
\label{4.1}\rho: H^0(A, L) \otimes H^0(A, L) \longrightarrow H^0(A,
L^{\otimes 2})
\end{equation}
(as given in (\ref{1.1})) is surjective. We are going to reduce
this to the question of vanishing of a certain cohomology group on
$\widetilde{A\times A}$ following the standard approach in [BEL,
Section 3].  Here $\pi: \widetilde{A\times A} \to A \times A$ is
the blow-up of $A \times A$ along the diagonal $D$ as in Section
\ref{Section 3}. Consider the short exact sequence on $A\times A$
given by
\begin{equation}
\label{4.2}0 \longrightarrow p_1^*L \otimes p_2^*L
\otimes \sI \longrightarrow p_1^*L \otimes p_2^*L \longrightarrow
p_1^*L \otimes p_2^*L\big|_{D} \longrightarrow 0,
\end{equation}
where
$\sI$ denotes the ideal sheaf of $D$.  Note that $p_1^*L \otimes p_2^*L\big|_D\cong L^{\otimes 2}$ under the natural isomorphism $D\cong A$, and one has $H^0(A\times A,p_1^*L \otimes p_2^*L)\cong H^0(A, L) \otimes H^0(A, L)$ by the K\"unneth formula.  Together with the long exact sequence associated to (\ref{4.2}), one easily sees that $\rho$ is surjective if $H^1(A \times A, p_1^*L \otimes p_2^*L \otimes \sI)
=0$.  But one also easily checks that
\begin{eqnarray}\quad
H^1(A \times A, p_1^*L \otimes p_2^*L \otimes \sI)&=& H^1(\widetilde{A \times A}, \sL\otimes \sO(- E))  \nonumber\\
\label{4.3}\qquad &=& H^1(\widetilde{A \times A}, K_{\widetilde{A \times A} }\otimes \sL\otimes \sO(- nE)),
\end{eqnarray}
where the last line follows from the isomorphism $K_{\widetilde{A \times A} }=\pi^*K_{A\times A}+\sO( (n-1)E)=\sO((n-1)E)$.  Finally if $\sL \otimes
\sO(-nE)$ is nef and big, then it follows from Kawamata-Viehweg vanishing theorem that
$H^1(\widetilde{A \times A}, K_{\widetilde{A \times A} }\otimes \sL\otimes \sO(- nE))=0$, which together with (\ref{4.3}), imply that $\rho$ is surjective.
\end{proof}

\medskip\noindent
\begin{lemma}\label{Lemma 4.2} Let $(A,L)$, $n$, $E$ and $\sL$ be as in Proposition \ref{Proposition 3.2}.
If $\sL
\otimes \sO(- n E)$ is nef and $L^n > (2n)^n$, then $\sL
\otimes \sO(- n E)$ is big. \end{lemma}

\begin{proof}  Note that $$ \sL^{2n} = (p_1^*L \otimes p_2^*L)^{2n}= \frac{(2n)!}{n! \cdot n!} L^n \cdot
L^n.$$ Recall that we have the identification $E = D \times
\mathbb P^{n-1}$ from the proof of Proposition 3.2. Denoting by
$\sigma: E \to \mathbb P^{n-1}$ and $\eta: E \to D = A$ the
projections, we have $\sO(E)|_E = \sigma^*\sO_{\mathbb
P^{n-1}}(-1)$ and $\sL|_E = \eta^*(L\otimes L)$.  From these,  a
straight-forward calculation gives
\begin{equation}
\label{4.4}  (\sL \otimes \sO(-nE))^{2n} = \frac{(2n)!}{n! \cdot n!} \cdot
L^n \cdot (L^n - (2n)^n).
\end{equation}
Together with the well-known fact that a nef line bundle is big if
and only if its top self-intersection number is positive, one
obtains the lemma readily
\end{proof}

\medskip
Finally we complete the proof of our main result as follows:

\medskip\noindent
{\it Proof of Theorem \ref{Theorem 1.3}}.  Let $\sA_{(d_1,\cdots,d_n)}$ denote the moduli space of $n$-dimensional polarized abelian varieties $(A,L)$ of a given polarization type $d_1|d_2|\cdots|d_n$ and satisfying
\begin{equation}
\label{4.5}  d_1\cdots d_n\geq \dfrac{8^n}{2}
\cdot \dfrac{n^n}{n!},
\end{equation}
and recall that
\begin{equation}
\label{4.6} h^0(A, L)= d_1\cdots d_n=\dfrac{L^n}{n!}\quad\textrm{for all }(A,L)\in \sA_{(d_1,\cdots,d_n)}.
\end{equation}
By [Ba, Theorem 1], there exists some $(A_o,L_o)\in \sA_{(d_1,\cdots,d_n)}$ such that
\begin{equation}
\label{4.7} m(A_o,L_o)= \dfrac{1}{\pi}\root n \of {2L_o^n}.
\end{equation}
Let $\sL_o$ be the line bundle over the blow-up $\widetilde{A_o \times A_o}$ of $A_o \times A_o$ along the diagonal (with exceptional divisor $E_o$) as in Proposition \ref{Proposition 3.2}.  From (\ref{4.5}), (\ref{4.6}) and (\ref{4.7}), one easily checks that $n\leq \frac{\pi}{8}\cdot m(A_o,L_o)$. Thus it follows from Proposition \ref{Proposition 3.2} that $\sL_o \otimes \sO(-nE_o)$ is nef.  One also easily checks from (\ref{4.5}) and (\ref{4.6}) that $L_o^n>(2n)^n$, and thus by Lemma \ref{Lemma 4.2}, the nef line bundle $\sL_o \otimes \sO(-nE_o)$ is also big.  Then it follows from Proposition \ref{Proposition 4.1} that $(A_o,L_o)$ is projectively normal.  Finally it is easy to see that the existence of a projective normal $(A_o,L_o)$ implies readily that a general $(A,L)$ in
$\sA_{(d_1,\cdots,d_n)}$ is projectively normal.
\qed

\enddocument